\documentclass[12pt,twoside]{amsart}
\usepackage{amsthm,amsfonts,amssymb,amscd,euscript}

\usepackage[matrix,arrow]{xy}

\author{Florin Ambro} 
\address{Institute of Mathematics ``Simion Stoilow'' of the Romanian
Academy\\
P.O. BOX 1-764, RO-014700 Bucharest\\ 
Romania.}
\email{florin.ambro@imar.ro}

\newcommand{\isoto}{{\overset{\sim}{\rightarrow}}}

\newcommand{\Q}{{\mathbb Q}}
\newcommand{\Z}{{\mathbb Z}}

\newcommand{\R}{{\mathbb R}}

\newcommand{\bP}{{\mathbb P}} % projective space
\newcommand{\bA}{{\mathbb A}} % affine space

\newcommand{\cO}{{\mathcal O}}

\newcommand{\bD}{{\mathbf D}}

\newcommand{\Conv}{\operatorname{Conv}}

\newcommand{\emb}{\operatorname{emb}}

\newcommand{\Fix}{\operatorname{Fix}}

\newcommand{\Int}{\operatorname{int}}

\newcommand{\lct}{\operatorname{lct}}

\newcommand{\length}{\operatorname{length}}

\newcommand{\mld}{\operatorname{mld}}
\newcommand{\gmld}{\operatorname{g-mld}}

\newcommand{\Mov}{\operatorname{Mov}}
\newcommand{\bMov}{\operatorname{{\mathbf Mov}}}
\newcommand{\mult}{\operatorname{mult}}

\newcommand{\relint}{\operatorname{relint}}

\theoremstyle{plain}
\newtheorem{thm}{Theorem}[section]

\newtheorem{lem}[thm]{Lemma}

\newtheorem{prop}[thm]{Proposition}

\theoremstyle{definition}

\newtheorem{exmp}[thm]{Example}

\newtheorem{rem}[thm]{Remark}

\theoremstyle{remark}

\begin{document}

\bibliographystyle{amsalpha+}
\title{Mld versus lct near zero}
%\date{October 11, 2024}
\maketitle

\dedicatory{
	\center{Dedicated to James McKernan on the occasion of his 60th birthday}	
	
}

\begin{abstract} 
We study the equivalence of approaching zero for two invariants of a singularity:
the minimal log discrepancy and the log canonical threshold of the 
general hyperplane section.
\end{abstract} 

%\tableofcontents

%%%%%%%%%%%%%%%%%%%%%%%%%%%%%%%%%%%%%%%%
%%% Document name: Bnd_Hyp.tex
%%% Last modified: October 11, 2024
%%%%%%%%%%%%%%%%%%%%%%%%%%%%%%%%%%%%%%%%



%%%%%%%%%%%%%%%%%%%%%%%%%%%%%%%%%%%%%%%%
%%%%%%%%%%%%%%%%%%%%%%%%%%%%%%%%%%%%%%%%

\section{Intro}

%%%%%%%%%%%%%%%%%%%%%%%%%%%%%%%%%%%%%%%%
%%%%%%%%%%%%%%%%%%%%%%%%%%%%%%%%%%%%%%%%

We consider in this note the following question: let $f\colon X\to Y\ni P$ be the germ 
of a Fano contraction, with positive dimensional base, near the fiber over a special point. 
Let $a=\mld_{f^{-1}P}X$ be the minimal log discrepancy of $X$
in the geometric valuations centered inside the central fiber. For a hyperplane section
$P\in H\subset Y$, let $\gamma(H):=\lct(X,f^*H)$ be the log canonical threshold of
its pullback, that is $\gamma \ge 0$ is maximal such that $(X,\gamma f^*H)$ has log canonical singularities.
If $\dim X$ is fixed, is $a$ bounded away from zero if and only if $\gamma:=\max_H\gamma(H)$ is bounded away from zero?

The converse implication follows from the inequality 
$
a\ge \gamma,
$
which holds since the Cartier closure of $f^*H$ has multiplicity at least one in each 
geometric valuation of $X$ centered inside $f^{-1}P$. The direct implication is non-trivial,
but there exists toric evidence.

The question emerged from our study of toric singularities (case $f$ isomorphism), while
trying 20 years ago to understand McKernan-Shokurov's conjecture on the singularities 
of Fano fibrations (solved~\cite{Bir23} in the meantime). Our hope was that 
Borisov's classification of toric $\Q$-factorial singularities into finitely many series~\cite{Bor99} 
would imply McKernan-Shokurov's conjecture, after finding a geometric interpretation for the 
series, which extends to the non-toric case.
Let $P\in X$ be a toric singularity of dimension $d$. If $d=2$, 
Borisov's classification has the following geometric 
interpretation:  the minimal log discrepancy is determined by the log canonical threshold of (one or 
two linearly independent) invariant hyperplane sections (see~\cite{Am24}, which is in fact a preprint 
from 2006). In particular, $\gamma \ge \frac{a}{2}$. If $d=3$, invariant hyperplane sections no longer 
determine the minimal log discrepancy. For example, the $\Q$-factorial toric singularity
$0\in \{z_4^2=z_1z_2z_3\}\subset \bA^4$ satisfies $K\sim 0$ and 
$$
\mld_0X=2>\frac{3}{2}=\max\{\sum_i t_i |\mld(X,\sum_it_iH_i)\ge 0, t_i\ge 0, H_i\text{ invariant hyperplane section} \}
$$
However, it is still true that $\gamma$ is bounded away from zero if so is $a$, as can be deduced  
from Borisov's classification (after replacing $\mld X$ by $\mld_PX$, and assuming $X$ is $\Q$-factorial). 
Borisov's classification extends to the non-$\Q$-factorial case, but only after enlarging the category of 
toric singularities back to the category of germs of toric Fano fibrations. In particular, our question has a positive answer 
in the toric case~\cite[Theorem 3.12]{Am22}. Here we also provide an effective estimate. 
An effective estimate also exists for (non-toric) surface singularities~\cite[Theorem 1.7]{Bin24}.

For  $d \in \Z_{ \ge 1}$ and $a\in (0,+\infty)$ , define recursively $\gamma(d,a)$ by
$\gamma(1,a)=a$ and $\gamma(d,a)=\gamma(d-1,\frac{a^2}{d^2})$ for $d\ge 2$.
One computes
$$
\gamma(d,a)=\frac{a^{2^{d-1}}}{\prod_{i=1}^d i^{2^{i-1}}}.
$$
In particular, $\gamma(d,a)$ is increasing in $a$. Recall that for a toric variety $X$,
the complement $\Sigma_X$ of the open dense torus inside $X$ satisfies 
$K+\Sigma_X\sim 0, \mld(X,\Sigma_X)=0$.

\begin{thm}\label{m1}
	Let $f \colon X \to Y\ni P$ be a toric contraction, where $Y$ is affine of positive
	dimension, containing an invariant point $P$. Suppose $0\le B\le \Sigma_X$ is an 
	invariant $\R$-divisor on $X$ such that  $-(K+B)$ is $f$-nef and 
	$\mld_{f^{-1}P}(X,B)=a>0$. 
	Then there exists an invariant hyperplane section $P\in H\subset Y$ such that 
	$(X,B+\gamma(d,a) f^*H)$ has log canonical singularities, where $d=\dim X$.
\end{thm}

\begin{exmp}
	Let $f\colon X\to \bA^1$ be a toric contraction. Let $X_0=f^*(0)$ be the special fiber,
	let $a>0$ such that $B=\Sigma_X-a X_0$ is effective.
	Then $K+B\sim_\R 0$, $\mld_{X_0}(X,B)=a$, $\mld_{X_0}(X,B+a X_0)=0$.
\end{exmp}

Let $n$ be a positive integer such that $n\gamma(d,a)\ge 1$. Set $B_n=\Sigma_X-\frac{1}{n}f^*H$.
Then 
$$
B\le B_n, n(K+B_n)\sim 0, \mld_{f^{-1}P}(X,B_n)\ge \frac{1}{n}.
$$ 
Therefore the toric case of Shokurov's conjecture~\cite{Sho04} on the existence of klt complements of
bounded index follows from Theorem~\ref{m1}. Moreover, the coefficients of $B$ are arbitrary.

Theorem~\ref{m1} holds in fact in a more general form (see Theorem~\ref{mn}), 
allowing a log Calabi-Yau structure with general boundary on $X$, and an immediate 
application is an effective version of the toric case of McKernan's conjecture:
if moreover $Y$ is $\Q$-factorial, then $\mld_PY\ge \gamma(d,a)$. 

McKernan's conjecture was solved in the toric case by Alexeev and Borisov~\cite{AB12}.
They use global minimal log discrepancies, but their method extends to minimal log
discrepancies over proper closed subsets $Z\subseteq Y$, to give an 
effective lower bound for $\mld_Z Y$ in terms of $\mld_{f^{-1}Z}X$, the dimension and the
$\alpha$-invariant of the 
general fiber of $f$. Our method is different, since we ignore the 
general fiber of $f$. In Theorem~\ref{mn} we work exclusively with the 
special fiber, also allowing log canonical centers which are not contained in the special fiber.  

The proof of Theorem~\ref{m1} is combinatorial, with the following geometric interpretation.
Let $d=\dim X$ and assume $Y$ is affine. 
We first construct a toric rational contraction $\phi\colon X\dashrightarrow \bP^1$ 
induced by an invariant pencil $\Lambda_1\subset |-n_1K-n_1B|$, such that $1\le n_1\le \frac{d^2}{t}$
and for a general member $D_1\in \Lambda_1$, the log Calabi-Yau structure of index $n_1$
$$
K+B+\frac{1}{n_1}D_1\sim_{n_1}0
$$ 
has log canonical singularities, and every lc center dominate $\bP^1$ via $\phi$. 
There are two possibilities.

a) Case $\mld_{f^{-1}P}(X,B+\frac{1}{n_1}D_1)>0$. This corresponds to an invariant
hyperplane section $P\in H\subset Y$ such that $(X,B+\frac{1}{n_1}f^*H)$ has log canonical singularities.
Here $\gamma \ge \frac{1}{n_1}\ge \frac{a}{d^2}$.

b) Case $\mld_{f^{-1}P}(X,B+\frac{1}{n_1}D_1)=0$. Let $\mu\colon X'\to X$ be the normalization
of the graph of $\phi$, let $\mu^*(K+B)=K_{X'}+B_{X'}$ be the log pullback. Let $X'_1$ be the
general fiber of $\phi'\colon X'\to \bP^1$, let $(K_{X'}+B_{X'})|_{X'_1}=K_{X'_1}+B_{X'_1}$
be the adjunction formula. The Stein factorization of the composition $X'_1\subset X'\to X\to Y$ induces a 
germ of toric contraction $f_1\colon X'_1\to Y_1\ni P_1$, with $P_1$ mapped onto $P$ and 
$\mld_{f_1^{-1}P_1}(X'_1,B_{X'_1})\ge \mld_{f^{-1}P}(X,B)$.
We have $\dim X'_1=d-1$ and the claim lifts from $X'_1$ to $X$, provided
$B_{X'_1}$ is effective. If not, the boundedness of $n_1$ allows to make
$B_{X'_1}$ effective after replacing it with a convex combination with the toric log canonical
complement $\Sigma_{X'_1}\in |-K_{X'_1}|$, at the expense of scaling down 
$\mld_{f_1^{-1}P_1}(X'_1,B_{X'_1})$ by a bounded factor. 

The natural category for the above argument is that of log varieties $(X,B)$, with $X$
toric and $B=\sum_i b_i D_i$ where $b_i\in [0,1]$ and $D_i$ a general member
of an invariant linear system $\Lambda_i$ on $X$. However, here we are concerned only
with the singularities of a general member of a linear system, not with the member itself.
Therefore it is more convenient to work with
toric generalized log varieties $(X,B+\bD_X)$, where $X$ is toric,
$B$ is effective invariant, and $\bD$ is an invariant $\R$-free $\R$-b-divisor. 

\begin{thm}\label{mn}
	Let $f \colon X \to Y\ni P$ be a toric contraction, where $Y$ is affine of positive
	dimension, containing an invariant point $P$. Let $(X,B+\bD_X)$ be a toric
	g-log variety structure such that $-(K+B+\bD_X)$ is $f$-nef and 
	$\gmld_{f^{-1}P}(X,B+\bD_X)=a>0$. Then there exists an invariant hyperplane 
	section $P\in H\subset Y$ such that $(X,B+\gamma(d,a) f^*H+\bD_X)$ has g-lc singularities,
	where $d=\dim X$.
\end{thm}

%\clearpage
%%%%%%%%%%%%%%%%%%%%%%%%%%%%%%%%%%%%%%%%
%%%%%%%%%%%%%%%%%%%%%%%%%%%%%%%%%%%%%%%%

\section{Preliminaries}

%%%%%%%%%%%%%%%%%%%%%%%%%%%%%%%%%%%%%%%%
%%%%%%%%%%%%%%%%%%%%%%%%%%%%%%%%%%%%%%%%

Let $X=T_N\emb(\Delta)$ be a toric variety. We use standard notation on toric varieties
(see~\cite{Oda88}). We always choose the anticanonical class to be $-K=\Sigma_X$,
the complement of the open dense torus sitting inside $X$.

%%%%%%%%%%%%%%%%%%%%%%%%%%%%%%%%%%%%%%%%

\subsection{Support functions}

%%%%%%%%%%%%%%%%%%%%%%%%%%%%%%%%%%%%%%%%

Let $M=N^*$ be the dual lattice. A subset $A\subset M_\R$ defines
a support function $h_A\colon N_\R \to \R \cup\{-\infty\}$ by $h_A(e)=\inf_{a\in A}\langle a,e\rangle$.
If $A$ has finite cardinality, the infimum is a minimum.
Let $t\in \R_{\ge 0}$. Then $tA=\{ta|a\in A \}\subset M_\R$ satisfies $h_{tA}=t\cdot h_A$.
The sum of two sets $A,A'\subset M_\R$ is $A+A':=\{a+a'| a\in A,a'\in A'\}\subset M_\R$,
and $h_{A+A'}=h_A+h_{A'}$.

%%%%%%%%%%%%%%%%%%%%%%%%%%%%%%%%%%%%%%%%

\subsection{Toric b-divisors}

%%%%%%%%%%%%%%%%%%%%%%%%%%%%%%%%%%%%%%%%
 A {\em toric valuation} of $X$ is the valuation
induced by an invariant prime divisor on a toric birational contraction $X'\to X$.
Toric valuations of $X$ are in bijection with the primitive lattice points contained
in the support of the fan defining $X$. Denote by $E_e$ the toric valuation corresponding to 
$e\in N^{prim}\cap |\Delta|$. The center of $E_e$ on $X$ is the orbit $O_\sigma$, where 
$\sigma\in \Delta$ is minimal with the property $e\in \sigma$.

Let $R\in \{\Z,\Q,\R\}$. A {\em toric $R$-b-divisor} of $X$ is a formal sum 
$$
\bD=\sum_{e\in N^{prim}\cap |\Delta|} d_e E_e\ (d_e\in R).
$$
The {\em trace of $\bD$ on a proper toric birational modification} $X'\dashrightarrow X$ 
is the invariant $R$-Weil divisor $\bD_X=\sum_{e \in \Delta_{X'}(1)} d_e V(e)$.
Traces are compatible with pushforwards. Conversely, a toric $R$-b-divisor is the same
as a pushforward compatible collection of invariant $R$-Weil divisors on the proper toric
birational modifications of $X$.

If $\bD_i$ is a finite collection of toric $R$-b-divisors of $X$, then 
$\min_i \bD_i:=\sum_e (\min_i d_{i,e}) E_e$ is a toric $R$-b-divisor of $X$.

An element $m\in M\otimes_\Z R$ induces the toric $R$-b-divisor $(\chi^m)=\sum_e \langle m,e\rangle E_e$
whose trace on $X'$ is the principal $R$-Weil divisor 
$(\chi^m)_{X'}=\sum_{e\in \Delta_{X'}(1)}  \langle m,e\rangle V(e)$.

A finite set $A\subset M\otimes_\Z R$ induces a toric $R$-b-divisor 
$\bD_A:=-\min_{a\in A}(\chi^a)=\sum_e-h_A(e)E_e$.

Let $L$ be an $R$-Cartier invariant $R$-Weil divisor on $X'$. The pullbacks of $L$
on higher models induce the $R$-b-divisor $\overline{L}$ of $X$, called the 
{\em Cartier closure of} $L$. 

Let $\bD$ be a toric $\R$-b-divisor of $X$. 
We say that $\bD$ {\em descends to $X'$} if $\bD_{X'}$ is $\R$-Cartier and $\bD=\overline{\bD_{X'}}$.
We say that $\bD$ is {\em $R$-free} if $\bD=\bD_A$ for some finite set $A\subset M\otimes_\Z R$.
Suppose $X\to Z$ is a proper toric morphism, with $Z$ affine. Then $\bD$ is $R$-free if and only if 
$\bD$ descends to a toric birational modification $X'\to X$ and $\bD_{X'}$ is $R$-Cartier and 
nef over $Z$. 

%%%%%%%%%%%%%%%%%%%%%%%%%%%%%%%%%%%%%%%%

\subsection{Invariant linear systems}

%%%%%%%%%%%%%%%%%%%%%%%%%%%%%%%%%%%%%%%%
Let $L$ be an invariant $\R$-Weil divisor on $X$, let $V\subset \Gamma(X,\cO_X(L))$
be an invariant $k$-vector space of finite positive dimension. Being invariant,
$V=\oplus_{a\in A}k\cdot \chi^m$ for a unique finite set $A\subset M\cap \square_L$.
The invariant linear system induced by $V\subset \Gamma(X,\cO_X(L))$ is 
$$
\Lambda=\{(f)_X+L | f\in V \setminus 0 \}.
$$
Since $\Lambda$ is invariant, so is the {\em fixed part} 
$\Fix\Lambda=\min_{D\in \Lambda}D$. One computes 
$$
\Fix \Lambda=\min_{a\in A}(\chi^a)_X+L.
$$
The {\em mobile part} 
$
\Mov\Lambda=\Lambda-\Fix\Lambda
$ 
coincides with the invariant mobile linear system on $X$ induced by  $V=\sum_{a\in A}k\cdot \chi^m\subset
\Gamma(X,\cO_X(L-\Fix \Lambda))$, where $L-\Fix\Lambda=(\bD_A)_X$. Therefore 
$\Mov\Lambda$ depends only on $A$, not on $L$.

The vector space $V$ induces its {\em mobile b-divisor} $\bMov V=-\min_{f\in V\setminus 0}(f)$,
whose restriction to toric modifications coincides with the b-free toric b-divisor $\bD_A$. 

Let $m\in M$, let $A'=A-m$ and $L'=L+(\chi^m)_X$. Then 
$V'=V\cdot \chi^{-m}$,  $V'\subset \Gamma(X,\cO_X(L'))$ induces the same linear system $\Lambda$,
and $\bMov V'=\bMov V+(\chi^m)$. When discussing properties invariant under linear equivalence,
we abuse notation and call $\bMov V$ the {\em mobile b-divisor of} $\Lambda$, denoted
$\bMov \Lambda$.

A finite set $A\subset M$ induces a pushforward compatible family of invariant mobile
linear systems $\Lambda_{X'}$ for all toric birational contractions $X'\to X$.
Indeed, on $X'$ it is induced by $L_{X'}=(\bD_A)_{X'}$ and 
$V=\sum_{a\in A}k\cdot \chi^a\subset \Gamma(X',\cO_{X'}(L_{X'}))$.
The linear system $\Lambda_{X'}$ has no base points if and only if the support function
$h_A$ is linear on each cone of $\Delta_{X'}$, and over such $X'$ it is also compatible
with pullbacks.

\begin{exmp}
Let $D$ be an effective invariant $\R$-Weil divisor on $X$. Set $A=\{0\}$ and $L=D$.
Then $A\subset M\cap \square_L$ induces the fixed linear system $\Lambda=\{D \}$,
with $\Fix \Lambda=D$ and $\bMov V=0$.
\end{exmp}

\begin{exmp}
	Let $\Lambda_i$ be the invariant linear system on $X$ induced by 
	$A_i\subset M\cap \square_{L_i}$, for $i=1,2$. The sum of linear systems
	$\Lambda_1+\Lambda_2$ is again invariant, induced by $A_1+A_2\subset M\cap \square_{L_1+L_2}$.
	Then $\Fix (\Lambda_1+\Lambda_2)=\Fix\Lambda_1+\Fix\Lambda_2$ and 
	$\bMov(V_1\cdot V_2)=\bMov V_1+\bMov V_2$. 
\end{exmp}

%%%%%%%%%%%%%%%%%%%%%%%%%%%%%%%%%%%%%%%%

\subsection{Length estimates for extension of non-negative linear maps}

%%%%%%%%%%%%%%%%%%%%%%%%%%%%%%%%%%%%%%%%
Let $\varphi\colon \Lambda\to \Z$ be a surjective homomorphism, where $\Lambda$
is a lattice of dimension at least two. Let $\sigma\subset \Lambda_\R$ be a rational polyhedral convex cone.
Let $C\subset \sigma$ be a convex set such that $\Lambda\cap C$ contains the origin and a set of generators of the cone $\sigma$. In particular, 
$
\sigma=\cup_{t>0}tC.
$
Suppose $\varphi(C)\subset \R$ is a compact interval of positive length $w$,
containing $0$ in its interior. 
Denote $\Lambda_0=\Lambda\cap \varphi^\perp$ and $C_0=C\cap \varphi^\perp\subset \Lambda_{0,\R}$.

\begin{prop}\label{extpofcn}
	Let $\varphi_0\colon \Lambda_0\to \Z$ be a lattice homomorphism such that
	$\varphi_0(C_0)=[0,l_0]$ for some $l_0>0$. Then there exists a lattice
	homomorphism $\varphi'\colon \Lambda\to \Z$ such that:
	\begin{itemize}
		\item $\varphi'|_{\Lambda_0}=q\varphi_0$ for some integer $1\le q< w$.
		\item $\varphi'(C)=[0,l']$ for some $0<l'\le wl_0$.
	\end{itemize}
\end{prop}

\begin{proof} The inclusion of lattices $j\colon \Lambda_0\subset \Lambda$ is a direct
	summand. It dualizes to a lattice projection $j^*\colon \Lambda^*\to \Lambda_0^*$,
	which identifies $\Lambda_0^*$ with the quotient $\Lambda^*/\Z \varphi$. 
	Let $C^*\subset \Lambda^*\otimes_\Z \R$ be the convex set polar dual to $C\subset \Lambda_\R$. 
	Let $C^*_0\subset \Lambda_0^*\otimes_\Z \R$ be the $j^*$-image of 
	$C\subset \Lambda^*\otimes_\Z \R$. Then $C_0\subset \Lambda_0\otimes_\Z \R$ is polar 
	dual to $C_0^*\subset \Lambda^*_0\otimes_\Z \R$. 
	
	Let $\varphi(C)=[-w_-,w_+]$, with $w_-,w_+>0$ and $w_-+w_+=w$.
	Denote $\varphi_1=\varphi$. Let $\varphi_2\colon \Lambda\to \Z$ be an arbitrary extension of 
	$\varphi_0$. Then $\varphi_1,\varphi_2\in \Lambda^*$ are linearly independent in $\Lambda^*_\R$.
	
	 By assumption, $\varphi_0$ is non-negative 
    on $\cup_{t>0}tC_0=\sigma\cap \varphi^\perp$. Equivalently, there exists $c\in \R$ such that 
	$$
	c\varphi_1+\varphi_2\in \sigma^\vee.
	$$
	By assumption, there exist finitely many elements $e_i \in \Lambda^{prim}\cap C$ such that
	$\sigma=\sum_i\R_{\ge 0}e_i\subset \Lambda_\R$. The condition for $c\varphi_1+\varphi_2\in \sigma^\vee$ 
	becomes
	$$
	\max_{\varphi_1(e_i)>0}\frac{-\varphi_2(e_i)}{\varphi_1(e_i)}\le c\le
	\min_{\varphi_1(e_i)<0}\frac{\varphi_2(e_i)}{-\varphi_1(e_i)}.
	$$
	Denote $\gamma'=\frac{1}{l_0}$. 
	Since $-\gamma'\varphi_0 \in C^*_0$, there exists $t\in \R$ such that 
	$t\varphi_1-\gamma'\varphi_2\in C^*$. Denote $c'=\gamma'c+t$. Then 
	$$
	c'\varphi_1-\gamma'(c\varphi_1+\varphi_2)\in C^*.
	$$
	
	{\em Step 1}: We claim that $-\frac{1}{w_+}\le c'\le \frac{1}{w_-}$. Indeed, from $\sigma^\vee\subseteq C^*$
	we deduce
	 $$
	c'\varphi_1=
	(c'\varphi_1-\gamma'(c\varphi_1+\varphi_2))+\gamma'(c\varphi_1+\varphi_2)
	\in C^*.
	$$
	The assumption $\varphi(C)=[-w_-,w_+]$ is polar dual to 
	$\{x\in \R | x \varphi_1\in C^* \}=[-\frac{1}{w_+},\frac{1}{w_-}]$. Therefore the claim holds.
	
	{\em Step 2}: We claim that $-\gamma'\frac{\min(w_-,w_+)}{w_-+w_+}\cdot (c\varphi_1+\varphi_2)\in C^*$.
	
	Indeed, suppose $c'\le 0$. Then 
	$$
	-\epsilon\gamma'(c\varphi_1+\varphi_2)=\epsilon(c'\varphi_1-\gamma'(c\varphi_1+\varphi_2))+(1-\epsilon)\frac{\varphi_1}{w_-}\in C^*,
	$$
	where $\epsilon=\frac{1}{1-c'w_-}\in (0,1]$. Since $c'\ge -\frac{1}{w_+}$, we deduce $\epsilon\ge  \frac{w_+}{w_-+w_+}$.
	
	Suppose $c'\ge 0$. Then 
	$$
	-\epsilon\gamma'(c\varphi_1+\varphi_2)=\epsilon(c'\varphi_1-\gamma'(c\varphi_1+\varphi_2))+(1-\epsilon)\frac{-\varphi_1}{w_-}\in C^*,
	$$
	where $\epsilon=\frac{1}{1+c'w_+}\in (0,1]$.
	Since $c'\le \frac{1}{w_-}$, we deduce $\epsilon\ge  \frac{w_-}{w_-+w_+}$.
	
	{\em Step 3}:
	Suppose $w_-\le w_+$. Choose $c=\min_{\varphi_1(e_i)<0}\frac{\varphi_2(e_i)}{-\varphi_1(e_i)}$.
	From Step 2 we have
	$$
	-\gamma'\frac{w_-}{w_-+w_+}\cdot (c\varphi_1+\varphi_2)\in C^*.
	$$
	There exists $e_i$ such that $\varphi_1(e_i)<0$ and $c=\frac{\varphi_2(e_i)}{-\varphi_1(e_i)}$.
	Here $\varphi_1(e_i)=-q$ for some positive integer $q$. Set 
	$\varphi'=\varphi_2(e_i)\varphi_1-\varphi_1(e_i) \varphi_2 \in \Lambda\setminus 0$.
	Then $\varphi'=q(c\varphi_1+\varphi_2)\in \sigma^\vee$, $j^*(\varphi')=q\varphi_0$ and
	$$
	-\gamma'\frac{w_-}{w_-+w_+}\frac{1}{q}\varphi' \in C^*.
	$$
	Since $e_i \in C$ and $\varphi_1(C)\subset [-w_-,w_+]$, 
	we obtain $q\le w_-$. Therefore $q<w$ and 
	$$
	\frac{w_-}{w_-+w_+}\frac{1}{q} \ge  \frac{1}{w}.
	$$
	We obtain $-\frac{\gamma'}{w}\varphi'\in C^*$. Then $\varphi'(C)$ is contained
	in $[0,\frac{w}{\gamma'}]=[0,wl_0]$.
	
	If $w_+\le w_-$ we choose $c$ to be the other extremal point, and argue similarly.
\end{proof}

%%%%%%%%%%%%%%%%%%%%%%%%%%%%%%%%%%%%%%%%
%%%%%%%%%%%%%%%%%%%%%%%%%%%%%%%%%%%%%%%%

\section{General(-ized) log pair structures on toric varieties}

%%%%%%%%%%%%%%%%%%%%%%%%%%%%%%%%%%%%%%%%
%%%%%%%%%%%%%%%%%%%%%%%%%%%%%%%%%%%%%%%%

Let $X=T_N\emb(\Delta)$ be a toric variety. We consider (generalized) log pair structures $(X,B)$
which can be resolved by toric modifications, and therefore log discrepancies 
and minimal log discrepancies can be computed combinatorially. Only the
invariant part of $B$ may have negative coefficients. 
If $B$ is effective, we say $(X,B)$ is a (generalized) log variety.

%%%%%%%%%%%%%%%%%%%%%%%%%%%%%%%%%%%%%%%%

\subsection{Invariant boundary}

%%%%%%%%%%%%%%%%%%%%%%%%%%%%%%%%%%%%%%%%

Let $B$ be an invariant $\R$-divisor
such that $K+B$ is $\R$-Cartier. The latter is equivalent to: for every 
$\sigma\in \Delta(top)$, there exists $\psi_\sigma\in M_\R$ such that 
$(\chi^{\psi_\sigma})+K+B=0$ on $U_\sigma$, that is 
$\langle \psi_\sigma,e_i\rangle=1-\mult_{V(e_i)}B$ for every $e_i\in \sigma(1)$.

Let $\mu\colon X'\to X$ be a toric birational contraction such that $X'$ is smooth.
Then $(X',\Sigma_{X'})$ is log smooth. Let $\mu^*(K+B)=K_{X'}+B_{X'}$ be the
log pullback. Then $B_{X'}$ is an invariant $\R$-divisor, with coefficients
$$
\mult_{V(e)}B_{X'}=1-\langle \psi_\sigma,e\rangle \ (e\in \Delta_{X'}(1),e\in \sigma).
$$
In other words, the log discrepancy of $(X,B)$ in a toric valuation $E_e\ (e\in N^{prim}\cap |\Delta|)$ 
of $X$, is computed as follows:
$$
a_{E_e}(X,B)=\langle \psi_\sigma,e\rangle\ (e\in \sigma\in \Delta(top)).
$$
In particular, $(X,B)$ has log canonical singularities if and only if $\psi_\sigma\in \sigma^\vee$ for every
$\sigma\in \Delta(top)$.

%%%%%%%%%%%%%%%%%%%%%%%%%%%%%%%%%%%%%%%%

\subsection{General boundary (Alexeev)}

%%%%%%%%%%%%%%%%%%%%%%%%%%%%%%%%%%%%%%%%

Let $X=T_N\emb(\Delta)$ be a toric variety, let $B=B^{inv}+\sum_j b_jS_j$,
where $B^{inv}$ is an invariant $\R$-Weil divisor on $X$, $b_j\ge 0$ and 
$S_j$ is a general member of an invariant mobile linear system $\Lambda_j$ on $X$.
Let $\Lambda_j$ be defined by a finite set $A_j\subset M$.
The condition that $K+B$ is $\R$-Cartier is equivalent to:
for every $\sigma\in \Delta(top)$, there exists $\psi_\sigma\in M_\R$ such that 
$$
\langle \psi_\sigma,e\rangle-h_{\sum_j b_jA_j}(e)=1-\mult_{V(e)}B^{inv}\ \forall e\in \sigma(1).
$$
Let $\mu\colon X'\to X$ be a toric birational contraction such that $X'$ is smooth,
and $\bMov \Lambda_j$ descends to $X'$ for all $j$.
Let $S'_j$ be a general member of the base point free linear system induced by $\Lambda_j$
on $X'$. Then $S'_j$ is smooth and $(X',\Sigma_{X'}+\sum_j S'_j)$ is log smooth.
We may take $S_j=\mu_*(S'_j)$. Then 
$$
\mu^*(K+B)=K_{X'}+\sum_i(1-a_i)E_i+\sum_jb_jS'_j,
$$
where $E_i=V(e_i)$ are the invariant prime divisors of $X'$ and 
$
a_i=\langle \psi_\sigma,e_i\rangle-h_{\sum_j b_jA_j}(e_i) \ (e_i\in \sigma).
$
In other words, the log discrepancy of $(X,B)$ in a toric valuation $E_e\ (e\in N^{prim}\cap |\Delta|)$ 
of $X$ is computed as follows:
$$
a_{E_e}(X,B)=\langle \psi_\sigma,e\rangle-h_{\sum_j b_jA_j}(e)   \ (e\in \sigma\in \Delta(top)).
$$
In particular, $(X,B)$ has log canonical singularities if and only if $b_j\le 1$ for all $j$, and 
$\psi_\sigma\in \Conv(\sum_jb_jA_j)+\sigma^\vee$ for every $\sigma\in \Delta(top)$.

The definition covers the case when the invariant linear systems have fixed parts.
Let $B=B^{inv}+\sum_jb_jD_j$ where $b_j\ge 0$ and $D_j$ is a general member of 
an invariant linear system of $\R$-Weil divisors on $X$. 
Let $\Lambda_j=F_j+\Lambda^m_j$ be the fixed-mobile decomposition.
Then $D_j=F_j+S_j$, where $S_j$ is a general member of the invariant mobile
linear system $\Lambda^m_j$ (zero if $\Lambda_j$ is fixed). Then 
$$
K+B^{inv}+\sum_jb_jD_j=K+(B^{inv}+\sum_j b_jF_j)+\sum_j b_jS_j.
$$

For each $j$, choose a positive integer $n_j$. Let 
$\tilde{\Lambda}_j=\Lambda_j+\cdots+\Lambda_j$ $(n_j$-times), let $\tilde{D}_j$
be the general member of $\tilde{\Lambda}_j$. Then 
$K+B^{inv}+\sum_jb_jD_j$ and $K+B^{inv}+\sum_j\frac{b_j}{n_j}\tilde{D}_j$ 
have the same log discrepancies in toric valuations.

%%%%%%%%%%%%%%%%%%%%%%%%%%%%%%%%%%%%%%%%

\subsection{Generalized boundary (Birkar-Zhang)}

%%%%%%%%%%%%%%%%%%%%%%%%%%%%%%%%%%%%%%%%

Let $B=B^{inv}+\sum_j b_jS_j$ where
$B^{inv}$ is an invariant $\R$-Weil divisor on $X$, $b_j\ge 0$ and $S_j$ is a general
member of an invariant mobile linear system $\Lambda_j$ on $X$.
Let $\bD$ be an $\R$-free invariant $\R$-b-divisor of $X$,  induced by a finite set
$A\subset M_\R$. Thus 
$$
\bD=\sum_{e\in N^{prim} \cap |\Delta| }-h_A(e)E_e.
$$
Let $\Lambda_j$ be induced by the finite set $A_j\subset M$.
We assume that $K+B+\bD_X$ is $\R$-Cartier. Equivalently, for every 
$\sigma\in \Delta(top)$ there exists $\psi_\sigma\in M_\R$ such that 
$(\chi^{\psi_\sigma})+K+B+\bD_X=0$ on $U_\sigma$, i.e. 
$$
\langle \psi_\sigma,e_i \rangle-h_{A+\sum_j b_jA_j}(e_i)=1-\mult_{V(e_i)}B^{inv}\ \forall e_i\in \sigma(1).
$$
Let $\mu\colon X'\to X$ be a toric birational contraction such that $X'$ is smooth,
$\bD$ and all $\bMov \Lambda_j$ descend to $X'$. Then
$$
\mu^*(K+B+\bD_X)=K_{X'}+\sum_i (1-a_i)E_i+\sum_jb_jS'_j+\bD_{X'},
$$
where 
$$
a_i=\langle \psi_\sigma,e_i\rangle-h_{A+\sum_jb_jA_j}(e_i)\ ( e_i\in \sigma ).
$$

In particular, $(X,B+\bD_X)$ has g-log canonical singularities if and only if 
$\psi_\sigma\in \Conv(A+\sum_j b_jA_j)+\sigma^\vee$ for every $\sigma$.

For $m\in M_\R$, denote $\bD'=\bD+(\chi^m)$. Then 
$(X,B+\bD_X)$ and $(X,B+\bD'_X)$ have the same g-log discrepancies in toric valuations.
Indeed, g-log discrepancies depend only on $\bD-\overline{\bD_X}$, which does not change
after replacing $\bD$ by $\bD'$.

Let $K+B^{inv}+\sum_jb_jD_j$ be a log pair structure with general boundary
induced by invariant linear systems $\Lambda_j$ on $X$. 
Let $\bD=\sum_j b_j\bMov \Lambda_j$.
Then $K+B^{inv}+\sum_j b_j\Fix\Lambda_j+\bD_X$ is a generalized log pair 
structure, having the same log discrepancies in toric valuations as 
$K+B^{inv}+\sum_jb_jD_j$. In particular, for an invariant closed subset $Z\subseteq \Sigma_X$
we have 
$$
\mld_Z(X,B^{inv}+\sum_jb_jD_j)=\gmld_Z(X,B^{inv}+\sum_j b_j\Fix\Lambda_j+\bD_X).
$$
Due to this property, when focusing only on the singularities of general members 
of linear systems, it is convenient to replace general boundaries by generalized ones.
The same holds for a combination of general and generalized pair: 
$$
(X,B^{inv}+\sum_jb_j\Lambda_j^{gen}+\bD_X) \leadsto
(X,B^{inv}+\sum_j b_j\Fix \Lambda_j+(\sum_jb_j\bMov \Lambda_j+\bD)_X).
$$

%\clearpage
%%%%%%%%%%%%%%%%%%%%%%%%%%%%%%%%%%%%%%%%
%%%%%%%%%%%%%%%%%%%%%%%%%%%%%%%%%%%%%%%%

\section{Anti-minimal models}

%%%%%%%%%%%%%%%%%%%%%%%%%%%%%%%%%%%%%%%%
%%%%%%%%%%%%%%%%%%%%%%%%%%%%%%%%%%%%%%%%

Let $f\colon X\to Y$ be a proper toric contraction, with $Y$ affine and containing an invariant point $P$. 
The contraction $f\colon X=T_N\emb(\Delta)\to Y=T_{\bar{N}}\emb(\bar{\sigma})$ corresponds to 
a lattice projection $\pi\colon N\to \bar{N}$ such that $|\Delta|=\pi^{-1}(\bar{\sigma})$.
In particular, $|\Delta|$ is a rational polyhedral convex cone.
Note that $Y$ contains an invariant point $P$ only if $\dim \bar{\sigma}=\dim \bar{N}$,
which is equivalent to $\dim |\Delta| =\dim N$. Moreover, $N$ and $\Delta$ determine
$f$, since $\bar{N}$ is the quotient of $N$ modulo $|\Delta|\cap (-|\Delta|)$, and 
$\bar{\sigma}$ is the image of $|\Delta|$.

Let $(X,B+\bD_X)$ be a g-log pair structure such that $-(K+B+\bD_X)$ 
is $f$-nef. Since $f$ is toric and $Y$ is affine, this is equivalent to $-(K+B+\bD_X)$ being 
$\R$-free. In other words, for every $\sigma\in \Delta(top)$, there exists 
$\psi_\sigma\in M_\R$ such that $(\chi^{-\psi_\sigma})-(K+B+\bD_X)$ is effective on $X$,
and zero on $U_\sigma$. In particular,
$\square_{-K-B-\bD_X}=\cap_{\sigma\in \Delta(top)}-\psi_\sigma+\sigma^\vee$.

Let $X=T_N\emb(\Delta)$, let $\bD$ be defined by a finite set $A\subset M_\R$. 
The polyhedral convex set
$$
\square:=\Conv(A)+\square_{-K-B-\bD_X}\subset M_\R
$$
does not depend on the choice of $A$.

\begin{lem}
	Let $e\in N^{prim}\cap |\Delta|$. The g-log discrepancy of $(X,B+\bD_X)$
	in the toric valuation $E_e$ equals $-h_\square(e)$.
\end{lem}

\begin{proof}
Let $\mu\colon X'\to X$ be a toric birational contraction such that $\bD$ descends to
$X'$. Let $e\in \sigma'\in \Delta_{X'}(top)$, let $\sigma'\subset \sigma\in \Delta(top)$.	
We have $(\chi^{\psi_\sigma})+K+B+\bD_X=0$ on $U_\sigma$.
Then $(\chi^{\psi_\sigma})_{X'}+K_{X'}+B_{X'}+\bD_{X'}=0$ on $U_{\sigma'}$.
There exists $a\in A$ such that $(\chi^a)+\bD_{X'}=0$ on $U_{\sigma'}$.
That is, $h_A(e')=\langle a,e'\rangle$ for every $e' \in \sigma'$.

Since $\bD$ descends to $X'$, the g-log discrepancy of $(X,B+\bD_X)$ in $E_e$
is the same as the log discrepancy of the log pair $(X',B_{X'})$.
We have $(\chi^{\psi_\sigma-a})_{X'}+K_{X'}+B_{X'}=0$ on $U_{\sigma'}$. Therefore
$$
a_{E_e}(X,B+\bD_X)=\langle \psi_\sigma-a,e\rangle.
$$
We have $\Conv(A)\subset a+\sigma'^\vee$. Since $-(K+B+\bD_X)$ is $\R$-free,
$\square_{-(K+B+\bD_X)}\subset -\psi_\sigma+\sigma^\vee$. We deduce
$$
a-\psi_\sigma\in \square\subset a-\psi_\sigma+{\sigma'}^\vee.
$$
Therefore 
$
a_{E_e}(X,B+\bD_X)=-h_\square(e).
$
\end{proof}

\begin{prop}\label{lccr}
	$(X,B+\bD_X)$ has g-lc singularities if and only if $\square$ contains the origin.
\end{prop}

\begin{proof}
	Since $\bD$ descends to a toric modification of $X$, the left hand side is 
	equivalent to $h_\square(e)\le 0$ for all $e\in N^{prim}\cap |\Delta|$,
	which is equivalent to the same property for all $e\in |\Delta|$. The latter
	is equivalent to $0\in \square$.
\end{proof}

\begin{lem}\label{sqc1}
	$\square=\{ m\in M_\R |  (X,B+(\chi^{-m})_X+\bD_X) \text{ has g-lc singularities}  \}$.
\end{lem}

\begin{proof}
	We have $\square_{-(K+B+(\chi^{-m})_X+\bD_X)}=\square_{-(K+B+\bD_X)}-m$.
	Since $\bD$ is the same, we conclude from Proposition~\ref{lccr}.
\end{proof}

\begin{lem}
	Let $\mu\colon X'\to X$ be a toric birational contraction, let
	$\mu^*(K+B+\bD_X)=K_{X'}+B_{X'}+\bD_{X'}$ be the log pullback.
	Then $\square\subseteq \square_{-K_{X'}-B_{X'}}$, and 
	equality holds if $\bD$ descends to $X'$.
\end{lem}

\begin{proof} We may suppose $\bD$ descends to $X'$.
	We have $\mu^*(K+B+(\chi^{-m})+\bD_X)=K_{X'}+B_{X'}+(\chi^{-m})_{X'}+\bD_{X'}$.
	Therefore $K+B+(\chi^{-m})_{X'}+\bD_X$ has g-lc singularities if and only if 
	$K_{X'}+B_{X'}+(\chi^{-m})_{X'}\le K_{X'}+\Sigma_{X'}=0$.
	The latter inequality is equivalent to $(\chi^m)_{X'}-K_{X'}-B_{X'}\ge 0$.
\end{proof}

Consider the polar dual convex polyhedral set 
$$
U=\square^*\subset N_\R.
$$ 
By definition, $U=\{e\in N_\R \ | -h_\square(e)\le 1 \}$.
Therefore $|\Delta|=\cup_{t>0}tU$. Then $\dim U=\dim |\Delta|$, so
$\dim U=\dim N$. We obtain
$$
\Int(tU)=\{e\in \Int |\Delta| \ | -h_\square(e)<t\} \ \forall t>0.
$$

\begin{prop}\label{mc}
	Let $t>0$. Then $\gmld_{f^{-1}(P)}(X,B+\bD_X)\ge t$ if and only if 
	$
	N\cap \Int(tU)\subset \{0\}.
	$
\end{prop}

\begin{proof} Since $f^{-1}P$ is invariant, $\gmld_{f^{-1}(P)}(X,B+\bD_X)$ is computed
	by toric valuations. Let $e\in N^{prim}\cap |\Delta|$, corresponding to the toric valuation $E_e$
	of $X$. Then $c_X(E_e)\subset f^{-1}P$ if and only if
	$\pi(e)\in \Int\bar{\sigma}$, which is equivalent to $e\in \Int |\Delta|$. We deduce that 
	$\gmld_{f^{-1}(P)}(X,B+\bD_X)\ge t$ if and only if $N^{prim}\cap \Int(tU)=\emptyset$.
	The latter is equivalent to $N\cap \Int(tU)\subset \{0\}$.
\end{proof}

Since $P\in Y$ is an invariant closed point, there are exactly two possibilities for $0\in U$:
\begin{itemize}
	\item[a)] Case $\dim Y=0$. Here $0\in \Int U$.
	Then $N\cap \Int(tU)\subset \{0\}$ if and only if $N\cap \Int(tU)= \{0\}$.
	\item[b)] Case $\dim Y>0$. Here $0\in \partial U$.
	Then $N\cap \Int(tU)\subset \{0\}$ if and only if $N\cap \Int(tU)= \emptyset$.
\end{itemize}

\begin{exmp}
	Suppose $K+B+\bD_X\sim_\R 0$. Choose $\psi\in M_\R$ with $(\chi^\psi)+K+B+\bD_X=0$.
	Then $\square_{-(K+B+\bD_X)}=-\psi+|\Delta|^\vee$, $\square=\Conv(A-\psi)+|\Delta|^\vee$ and 
	$a_{E_e}(X,B+\bD_X)=\langle \psi,e\rangle-h_A(e)$ for all $e\in N^{prim}\cap |\Delta|$.
\end{exmp}

The polyhedral convex set $U$ contains $0$, and can be decomposed as a Minkowski sum of
a polytope and a convex cone $\sigma_0$. We compute
$$
\sigma_0=\{e\in N_\R | h_\square(e) \ge 0 \} \subset |\Delta|.
$$
In particular, $(X,B+\bD_X)$ has g-klt singularities if and only if 
$U$ is compact, if and only if the origin lies in the relative interior of $\square$.
Note that $\gmld_{f^{-1}P}(X,B+\bD_X)>0$ if and only if the origin belongs to $\square$,
but does not belong to any compact face of $\square$.

\begin{exmp}
	$(X,B+\bD_X)$ has log discrepancy zero in all toric valuations if and only if 
	$B=\Sigma_X$ and $\bD=(\chi^m)$ for some $m\in M_\R$. Here $K+B+\bD_X\sim_\R 0$, 
	$\square=|\Delta|^\vee$ and $U=|\Delta|$.
\end{exmp}

\begin{lem}
	Let $\mu\colon X'\to X$ be a toric birational contraction such that $\bD$ descends to $X'$.
	Denote $\Delta_{X'}(1)=\{e_i\}$ and $B_{X'}=\sum_i(1-a_i)V(e_i)$. Then 
	$$
	U=\Conv(\{0\}\cup\{ \frac{e_i}{a_i} | a_i>0\})+\sum_{a_i=0}\R_{\ge 0}e_i.
	$$
	In particular, $\sigma_0\subset N_\R$ is a rational polyhedral cone, and
	the extremal points of $U$ are contained in the set $\{0\}\cup \{ \frac{e_i}{a_i} | a_i>0\}$.
\end{lem}

\begin{proof}
The polar dual of the right hand side equals $\cap_i \{m\in M_\R | \langle m,e_i\rangle+a_i\ge 0\}$,
which is just $\square_{-K_{X'}-B_{X'}}$. The latter equals $\square$, since $\bD$ descends to 
$X'$. By duality, the right hand side equals $\square^*=U$.
\end{proof}

Is it true that $0$ is an extremal point of $U$ iff $0\prec |\Delta|$?

Denote $l:=\dim N-\dim \sigma_0\ge 0$. Note that $l=0$ if and only if $\gmld_{f^{-1}(P)}(X,B+\bD_X)\le 0$.

\begin{lem}\label{ft} 
	Suppose $l\ge 1$, $t>0$ and $N\cap \Int(tU)=\emptyset$.
Consider the lattice projection $p\colon N\to N'=N/N\cap(\sigma_0-\sigma_0)$.
Let $U' \subset N'_\R$ be the $p$-image of $U\subset N_\R$. Then:
\begin{itemize}
	\item[a)] $N'\cap \Int(tU')=\emptyset$.
	\item[b)] There exists a surjective linear homomorphism $\varphi\colon N\to \Z$ with $\length \varphi(U)\le \frac{l^2}{t}$. 
\end{itemize}
\end{lem}

\begin{proof}
	a) Suppose $e'\in N'\cap \Int(tU')$. Then $e'=pe=pv$, where $e\in N$ and $v\in \Int(tU)$.
	Then $e-v=v_1-v_2$ for some $v_1,v_2\in \sigma_0$. We obtain $e=v+v_1-v_2$.
	Let $\sigma_0=\R_+e_i$ with $e_i\in N$. We can write $v_2=\sum_it_ie_i$ with $t_i\ge 0$ for all $i$.
	Then $e+\sum_i\lceil t_i\rceil e_i=v+v_1+\sum_i(\lceil t_i\rceil-t_i)e_i$. The left hand side belongs to
	$N$ and the right hand side belongs to $\Int(tU)$. Then $N\cap \Int(tU)\ne \emptyset$,
	a contradiction!
	
	b) Note that $l=\dim N'$. From a) and~\cite[Theorem 5.2]{AI20}, $w(N',tU')\le l^2$. 
	That is, there exists a surjective linear homomorphism $\varphi' \colon N' \to \Z$ with 
	$\length \varphi'(U')\le \frac{l^2}{t}$. Set $\varphi=\varphi'\circ p\colon N\to \Z$.
	Then $\varphi\colon N\to \Z$ is surjective and $\varphi(U)=\varphi'(U')$. 
\end{proof}

%\clearpage
%%%%%%%%%%%%%%%%%%%%%%%%%%%%%%%%%%%%%%%%
%%%%%%%%%%%%%%%%%%%%%%%%%%%%%%%%%%%%%%%%

\section{A construction}

%%%%%%%%%%%%%%%%%%%%%%%%%%%%%%%%%%%%%%%%
%%%%%%%%%%%%%%%%%%%%%%%%%%%%%%%%%%%%%%%%

Let $f\colon (X,B+\bD_X)\to Y\ni P$ be an anti-minimal model as in Section 4.
Suppose moreover that $\dim Y>0$ and $\gmld_{f^{-1}(P)}(X,B+\bD_X)\ge t>0$. In particular, 
$(X,B+\bD_X)$ has g-lc singularities.

Let $\varphi\colon N\to \Z$ a surjective linear homomorphism such that $\varphi(U)\subset \R$
is compact. Then $\varphi(U)$ is an interval of finite length $w>0$. Since $U$ contains the origin,
so does $\varphi(U)$. Modulo replacing $\varphi$ by $-\varphi$, we have exactly two 
cases: 1) $\varphi(U)=[0,w]$ for some $w>0$, or 2) $\varphi(U)=[-a,b]$,  with $a,b>0$ and 
$a+b=w$.

{\em Case 1):} The condition $\varphi(U)=[0,w]$ is dual equivalent to 
$$
\{x\in \R\  | x\varphi\in \square \}=[-\frac{1}{w},+\infty).
$$
We have $\varphi\in |\Delta|^\vee$.
Therefore there exists $\bar{\varphi}\in \bar{M}^{prim}\cap {\bar{\sigma}}^\vee$ such that
$\pi^*(\bar{\varphi}) =\varphi$. Then $H=(\chi^{\bar{\varphi}})_Y$ is an invariant hyperplane section 
of $P\in Y$ such that $f^*H=(\chi^{\varphi})_X$. Therefore
$$
(X,B+\frac{1}{w} f^*H+\bD_X)
$$ 
has g-lc singularities (and $\frac{1}{w}$ is maximal with this property).

{\em Case 2):} The condition $\varphi_1(U)=[-a,b]$ is dual equivalent to 
$$
\{x\in \R \ | x\varphi \in \square \}=[-\frac{1}{b},\frac{1}{a}].
$$
In particular, $\pm \frac{1}{w}\varphi \in \square$, which implies
$
-h_\square(e)\ge \frac{|\langle \varphi ,e\rangle |}{w}
$
for all $e\in |\Delta|$.
It follows that every $e\in N^{prim}\cap |\Delta|$ with $a_{E_e}(X,B+\bD_X)<\frac{1}{w}$
must satisfy $\varphi(e)=0$.

Let $M_0=M/\Z\varphi$ be the quotient lattice, dual to $N_0=N\cap \varphi^\perp$.
Let $\square_0$ be the image of $\square$ under the projection 
$M_\R\to M_{0,\R}$. The polar dual convex set $U_0=(\square_0)^*\subset N_{0,\R}$ equals $U\cap \varphi^\perp$.
The assumption $\gmld_{f^{-1}(P)}(X,B+\bD_X)\ge t$ is equivalent to 
$N\cap \Int(tU)=\emptyset$. We have $\relint(tU_0)\subset \Int(tU)$, since 
$\varphi(U)$ contains $0$ in its interior. We deduce
$$
N_0 \cap \Int(t U_0)=\emptyset.
$$

The projection $\varphi\colon N\to \Z$ corresponds to a fibration of tori
$T_N\to T_\Z$. Identifying $T_\Z$ with the standard open chart $\{z_0=1\}\subset \bP^1$,
we obtain a toric rational fibration 
$$
\phi \colon X\dashrightarrow  \bP^1, \ \phi (x)=[1:\chi^\varphi(x)].
$$ 
It coincides with the map induced by the invariant linear
system $\Lambda$ defined by the finite set $\{0,\varphi\}\subset M$. The support function
$h\colon N_\R\ni e\mapsto \min(0,\langle \varphi,e\rangle)\in \R$ is $N$-rational
and piecewise linear.  The rational map $\phi$ extends to a morphism from $X$ 
if and only if the functional 
$$
|\Delta|\ni e\mapsto \min(0, \langle \varphi,e\rangle)
$$
is linear on every cone of $\Delta$. 
Let $\Delta'$ be the collection of cones 
$\sigma\cap \varphi^{\ge 0}, \sigma\cap \varphi^{\le 0}, \sigma\cap \varphi^\perp$, after all $\sigma\in \Delta$.
Then $\Delta'$ is an $N$-rational fan, a subdivision of $\Delta$. It induces a toric birational contraction 
$\mu\colon X'=T_N\emb(\Delta') \to X$ which resolves  $\phi$ to a toric morphism:
\[ \xymatrix{
	& X' \ar[dl]_\mu \ar[dr]^{\phi'} & \\
	X\ar@{-->}[rr]^\phi  & & \bP^1
} \]
Note that $\phi'(X')=\bP^1$, since $\varphi(N)=\Z$ and $|\Delta|^\vee$
does not contain $\varphi$ or $-\varphi$. 
Let 
$$
\mu^*(K_X+B+\bD_X)=K_{X'}+B_{X'}+\bD_{X'}
$$ 
be the log pullback. The invariant $\R$-divisor
$B_{X'}$ satisfies $B_{X'} \le \Sigma_{X'}$, but it may not be effective in $\mu$-exceptional divisors.
We bound from below the negative coefficients of $B_{X'}$.

\begin{lem}\label{1r}
	The one dimensional cones of $\Delta'$ are the one dimensional cones of $\Delta$,
	plus the rays 
	$
	\R_+(\varphi(e_2) e_1-\varphi(e_1) e_2 ),
	$
	after all pairs $e_1,e_2$ which generate a two-dimensional cone in $\Delta$ and satisfy
	$ \varphi(e_1)<0< \varphi(e_2)$.
\end{lem}

\begin{proof}
	Indeed, suffices to determine the extremal rays of $\sigma\cap \varphi^{\ge 0}$, for 
	a top cone $\sigma\in \Delta$. By duality, we have $(\sigma\cap \varphi^{\ge 0})^\vee=
	\sigma^\vee+\R_+\varphi$. Let $\sigma^\vee(1)=\{g_j;j\}$ and $d=\dim N$.
	Suppose $g_1,\ldots,g_{d-1}$ are linearly independent elements of $\sigma^\vee(1)$.
	The space of solutions $\{g_1=\cdots=g_{d-1}=0\}$ is a one-dimensional vector space.
	If one generator is contained in $\sigma\cap \varphi^{\ge 0}$, we obtain an extremal ray of 
	$\sigma$ which is also an extremal ray of $\sigma\cap \varphi^{\ge 0}$.
	Suppose now $g_1,\ldots,g_{d-2}$ are elements of $\sigma^\vee(1)$ such that
	$g_1,\ldots,g_{d-2},\varphi$ are linearly independent. Let $v$ be a generator of the 
	one dimensional vector space of solutions $\{g_1=\cdots=g_{d-2}=\varphi=0\}$. 
	Then $v$ belongs to $\sigma\cap\varphi^{\ge 0}$ if and only if $v$ belongs to 
	$\tau=\sigma\cap\{g_1=\cdots=g_{d-2}=0\}$, a face of $\sigma$ of dimension at most two.
	If $\dim \tau=1$, then $\R_+ v$ is an extremal ray of $\sigma$ contained in $\varphi^\perp$.
	If $\dim \tau=2$, then we may order $\tau(1)=\{e_1,e_2\}$ such that 
	$\varphi(e_1)<0<\varphi(e_2)$.
	Since $\varphi(e_2) e_1-\varphi(e_1) e_2 \in \relint \tau\cap \varphi^\perp$,
	we obtain $\R_+v=\R_+( \varphi(e_2) e_1- \varphi(e_1) e_2 )$.
	By the proof of Farkas's Theorem, all extremal rays of $\sigma\cap \varphi^{\ge 0}$ are obtained this way.
	The claim is proved.
\end{proof}

For simplicity, denote by $a_e$ the g-log discrepancy
of $(X,B+\bD_X)$ in the toric valuations $E_e$ of $X$ corresponding to
$e\in N^{prim}\cap |\Delta|$. 

\begin{lem}\label{1w}
	$\max(a,b)\ge 1$. In particular, $w>1$.
\end{lem}
	
	\begin{proof}
		Since $\varphi(U)$ is compact and $U$ is the Minkowski sum of a 
		polytope and the cone $\sigma_0$, it follows that $\varphi$ is zero on $\sigma_0$.
		Since $\varphi\ne 0$ and $\Delta(1)$ generates $N_\R$, there exists 
		$e_j\in \Delta(1)$ such that $\varphi(e_j)\ne 0$. Suppose by contradiction that 
		$a_{e_j}=0$. Then $e_j\in \sigma_0$, hence $\varphi(e_j)=0$, a contradiction.
		We obtain $a_{e_j}>0$. Then $\frac{e_j}{a_{e_j}}\in U$. Therefore 
		$$
		-aa_{e_j}\le \varphi (e_j)\le ba_{e_j}.
		$$
		The integer $\varphi(e_j)$ is non-zero. If positive, we obtain $1\le ba_{e_j}\le b$.
		If negative, we obtain $1\le aa_{e_j}\le a$. 
	\end{proof}
	
\begin{lem}\label{bnda}
	$a_E(X,B+\bD_X)\le w$ for every invariant prime divisor $E$ on $X'$.
\end{lem} 

\begin{proof} 
	
	Suppose $e\in \Delta(1)$. Then $a_e\le 1$, since $B\ge 0$. By Lemma~\ref{1w}, $a_e<w$.
		
	Suppose $e'\in \Delta'(1)\setminus \Delta(1)$. Since $\varphi(N)\subset \Z$,
	there exists by Lemma~\ref{1r} $q=q(e')\in \Z_{\ge 1}$ such that 
	$$
	e'=\frac{ \varphi(e_2) e_1-\varphi(e_1) e_2}{q},
	$$
	where $e_1,e_2\in \Delta(1)$ are the generators of a $2$-dimensional cone in $\Delta$ and 
	$\varphi(e_1)<0< \varphi(e_2)$.
	Since $a_{E_e}(X,B+\bD_X)=-h_\square(e)$ and $h_\square(v+v')\ge h_\square(v)+h_\square(v')$, we have 
	$$
	a_{e'}\le \frac{\varphi(e_2)}{q}a_{e_1}+\frac{-\varphi(e_1)}{q}a_{e_2}.
	$$
	From $-\frac{\varphi}{b}\in \square$, that is $\max \varphi(U)\le b$,
	we obtain $ \varphi(e_2) \le ba_{e_2}$. From $\frac{\varphi}{a}\in \square$, that is 
	$\max -\varphi(U)\le a$, we obtain $ -\varphi(e_1) \le aa_{e_1}$. Therefore
	$$
	a_{e'}\le \frac{wa_{e_1}a_{e_2}}{q}
	$$
	Since $B$ is effective and $e_i\in \Delta(1)$, we have $a_{e_i}\le 1$ for $i=1,2$.
	Therefore $a_{e'}\le w$.
\end{proof}

Next, we show that $(X',B_{X'}+\bD_{X'})\to Y\ni P$ induces a similar structure on 
the general fiber of $\phi'$. Denote $F={\phi'}^{-1}(1)$.
Denote $X'_1=T_{N_0}\emb(\Delta_0)$, where $N_0$ is the kernel of $\varphi\colon N\to \Z$
and 
$$
\Delta_0=\{\sigma'  | \sigma'\in \Delta', \sigma'\subset N_{0,\R} \}=\{\sigma\cap \varphi^\perp  | \sigma\in \Delta \}.
$$
Note that $|\Delta_0|=|\Delta|\cap \varphi^\perp$.
The natural toric morphism $j\colon X'_1 \to X'$ induces an isomorphism $X'_1 \isoto F$. In fact,
a splitting of the projection $\varphi\colon N\to \Z$ induces a (non-canonical) toric isomorphism
\[ \xymatrix{
	X'_1 \times T \ar[dr]_{\text{pr}_2}   \ar[rr]^\simeq &   &   {\phi'}^{-1}(T)  \ar[dl]^{\phi'}  \\
	 & T & 
} \]
where $T=T_\Z$ is the open dense torus inside $\bP^1$. In particular, 
the toric valuations of ${\phi'}^{-1}(T)$ identify with the toric valuations 
of $X'_1$, and invariant divisors on ${\phi'}^{-1}(T)$ identify with the preimage
under the first projection of invariant divisors on $X_0$. Therefore 
$(K_{X'}+B_{X'}+\bD_{X'})|_{  {\phi'}^{-1}(T) }$ identifies with 
the product of g-log canonical classes
$$
(K_{X'_1}+B_{X'_1}+\bD'_{X'_1}) \boxtimes K_T,
$$
where $K_T\sim 0$, $B_{X'_1}$ is invariant, and $\bD'$ is the $\R$-free toric b-divisor of $X'_1$ induced by $A_0 \subset M_{0,\R}$,
the image of $A\subset M_\R$ under the projection $M\to M_0$.
Since $-(K+B+\bD_X)$ is $\R$-free, so is $-(K_{X'_1}+B_{X'_1}+\bD'_{X'_1})$.
If $B_{X'}$ is effective, then $B_{X'_1}$ is effective.
For every $e\in N_0^{prim}\cap |\Delta_0|=N^{prim}\cap |\Delta|\cap \varphi^\perp$, we have 
$$
a_{E_e}(X,B+\bD_X)=-h_\square(e)=-h_{\square_0}(e)=a_{E_e}(X'_1,B_{X'_1}+\bD'_{X'_1}).
$$
Therefore $\square(K_{X'_1}+B_{X'_1}+\bD'_{X'_1})=\square_0$
and $U(K_{X'_1}+B_{X'_1}+\bD'_{X'_1})=U_0$.

Denote $\bar{N}_0=\pi(N_0)\subset \bar{N}$, $\bar{\sigma}_0=\bar{N}_{0,\R}\cap \bar{\sigma}$
and $Y_1=T_{\bar{N}_0}\emb(\bar{\sigma}_0)$. We have a commutative diagram of toric morphisms
\[ \xymatrix{
	X'_1  \ar[d]_{f'_1}   \ar[r]^j   &   X'  \ar[d]^{f'}  \\
	Y_1 \ar[r]^u   & Y
} \]
where $f'=f\circ \mu\colon X'\to Y$, $f'_1$ is the contraction induced by $N_0\to \pi(N_0)$, 
and $u$ is finite over its image induced by $\pi(N_0)\subset \bar{N}$. Therefore 
$f'_1$ is isomorphic to the contraction component of the Stein factorization of 
$F\to f'(F)$. 

\begin{lem}
	$Y_1$ contains an invariant point $P_1$, and $u(P_1)=P$.
\end{lem}

\begin{proof}
	Since $0\in \Int \varphi(U)$, $|\Delta|^\vee$ does not contain $\varphi$ or $-\varphi$.
	Therefore 
	$$
	|\Delta|\cap\varphi^\perp-|\Delta|\cap \varphi^\perp=\varphi^\perp=N_{0,\R}.
	$$
	Since $|\Delta|=\pi^{-1}(\bar{\sigma})$, we have $\bar{\sigma}_0=\pi(|\Delta|\cap \varphi^\perp)$.
	Therefore $\bar{\sigma}_0-\bar{\sigma}_0=\bar{N}_{0,\R}$, that is $Y_1$ contains an invariant point.
	
	Let $P_1$ be the (unique) invariant point of $Y_1$. Since $u$ is toric, $u(P_1)$ is 
	an invariant point of $Y$, hence $u(P_1)=P$.
\end{proof} 

In particular, the image $f'(F)$ contains $P$. We have constructed the data 
$$
f_1\colon (X'_1,B_{X'_1}+\bD'_{X'_1})\to Y_1\ni P_1,
$$ 
with $\dim X'_1=\dim X-1$, 
$-(K_{X'_1}+B_{X'_1}+\bD'_{X'_1})$ $\R$-free and $\mld_{{f_1}^{-1}P_1}(X'_1,B_{X'_1}+\bD'_{X'_1})\ge t$.
But $B_{X'_1}$ may not be effective!

Fix $0\le \lambda \le \frac{1}{w}$. Denote $\tilde{B}=(1-\lambda)\Sigma_X+\lambda B$.
Then 
$$
K+\tilde{B}+\lambda \bD_X =(1-\lambda)(K+\Sigma_X)+\lambda (K+B+\bD_X)=\lambda(K+B+\bD_X).
$$
Then $\mu^*(K+\tilde{B}+\lambda \bD_X)=K_{X'}+\tilde{B}_{X'}+\lambda \bD_{X'}$, where 
$\tilde{B}_{X'}=(1-\lambda)\Sigma_{X'}+\lambda B_{X'}$. By Lemma~\ref{bnda},
$$
\tilde{B}_{X'}=\sum_{e'\in \Delta'(1)}(1-\lambda a_{e'})V(e')\ge 0.
$$
The restriction of $K_{X'}+\tilde{B}_{X'}+\lambda \bD_{X'}$ to $\phi^{-1}(T)$ identifies with
$(K_{X'_1}+\tilde{B}_{X'_1}+\lambda \bD'_{X'_1})\boxtimes K_T$, where 
$\tilde{B}_{X'_1}=(1-\lambda)\Sigma_{X'_1}+\lambda B_{X'_1}$ is effective.

\begin{lem}\label{exthyp}
	Suppose $(X_1,\tilde{B}_{X'_1}+\gamma_1{f_1}^*H_1+\lambda \bD_{X'_1})$ has g-lc
	singularities, where $P_1\in H_1\subset Y_1$ is an invariant hyperplane section
	and $\gamma_1>0$. Then there exists an invariant hyperplane section 
	$P\in H\subset Y$ such that $u^*H=qH_1$ for some integer $1\le q<w$, and 
	$(X,B+\frac{\gamma_1}{\lambda w}f^*H+\bD_X)$ has g-lc singularities.
\end{lem}

\begin{proof}
	There exists $\bar{\varphi}_0 \in \bar{M}_0 \cap \bar{\sigma}_0^\vee\setminus 0$ with 
	$H_1=(\chi^{\bar{\varphi}_1})_{Y_1}$. Consider the commutative diagram
	of lattice homomorphisms
	\[ \xymatrix{
		N_0  \ar[d]_{\pi_0}   \ar[r]^j   &   N  \ar[d]^\pi  \\
		\bar{N}_0 \ar[r]^u   & \bar{N}
	} \]
	Let $\varphi_0=\pi_0^*(\bar{\varphi}_0)\in M_0\cap |\Delta_0|^\vee\setminus 0$.
	Denote $\tilde{U}=U(X',\tilde{B}_{X'}+\lambda \bD_{X'})$. Then 
	$\tilde{U}=\lambda^{-1}U$ and $\tilde{U}_0=\lambda^{-1}U_0$.
	By assumption, $\varphi_0(\tilde{U}_0)\subset [0,\gamma_1^{-1}]$.
	Note that the cone $\cup_{t>0} t\tilde{U}$ equals $|\Delta|=|\Delta'|$, which is 
	generated by $\Delta'(1)$. Since $\tilde{B}_{X'}$ is effective, 
	$\Delta'(1)\subset \tilde{U}$. Therefore Proposition~\ref{extpofcn}
	applies: there exists $\varphi\in M\setminus 0$ such that $j^*(\varphi)=q\varphi_0$
	for some integer $1\le q <w$, and $\varphi(\tilde{U})\subset [0,w\gamma_1^{-1}]$.
	The latter is equivalent to
	$$
	\varphi(U)\subset [0,\lambda w\gamma_1^{-1}].
	$$
	Since $\varphi\in |\Delta|^\vee$ and $\pi$ is surjective, there exists 
	$\bar{\varphi}\in \bar{M}\cap\bar{\sigma}^\vee\setminus 0$ such that $\varphi=\pi^*(\bar{\varphi})$.
	Since $\pi_0$ is surjective, a diagram chase gives 
	$$
	q\bar{\varphi}_0=u^*(\bar{\varphi}).
	$$
	Therefore $H=(\chi^{\bar{\varphi}} )_Y$ satisfies the claim.
\end{proof}

\begin{rem}
	Let $n=\lceil w\rceil$. Then $\{-\varphi,\varphi\} \subset  M\cap n\square\subset M\cap \square_{-nK-nB-n\bD_X}$.
	Let $\Lambda_n \subset |-n(K+B+\bD_X)|$ be the induced invariant pencil. The toric rational map
	$$
	\phi_{\Lambda_n}\colon  X\dashrightarrow \bP^1, \ x \mapsto [\chi^{-\varphi}(x): \chi^{\varphi}(x)]
	$$
	has Stein decomposition $u\circ \phi$, where $\phi$ is our rational fibration and 
	$u\colon \bP^1\to \bP^1$ is the double cover $[z_0:z_1]\mapsto [z_0^2:z_1^2]$.
	Let $D_n$ be a general member of 
	$\Lambda_n$. Then $n(K_X+B+\bD_X+\frac{1}{n}D_n)\sim 0$, and for a toric valuation 
	$E_e$ of $X$ we compute its g-log discrepancy 
	$$
	a_{E_e}(X,B+\bD_X+\frac{1}{n}D_n)=-\min(\frac{\langle -\varphi,e\rangle}{n},\frac{\langle \varphi,e\rangle}{n} )=
	\frac{|\langle \varphi,e\rangle|}{n}\ge 0.
	$$
	Therefore $(X,B+\bD_X+\frac{1}{n}D_n)$ has g-lc singularities, and its toric lc places are 
	$$
	\{E_e | e\in N^{prim}\cap |\Delta|\cap \varphi^\perp \},
	$$
	that is the toric valuations of $X$ which are horizontal with respect to $\phi_{\Lambda_n}$.
	
\end{rem}

%%%%%%%%%%%%%%%%%%%%%%%%%%%%%%%%%%%%%%%%
%%%%%%%%%%%%%%%%%%%%%%%%%%%%%%%%%%%%%%%%

\section{Bounded hyperplanes}

%%%%%%%%%%%%%%%%%%%%%%%%%%%%%%%%%%%%%%%%
%%%%%%%%%%%%%%%%%%%%%%%%%%%%%%%%%%%%%%%%

Let $f\colon (X,B+\bD_X)\to Y\ni P$ and $a>0$ be a data as in Theorem~\ref{mn}.
Let $X=T_N\emb(\Delta)$, let $\sigma_0\subset N_\R$ be the cone of g-lc places.
Denote $l=\dim N-\dim \sigma_0$. Since
$$ 
l\le \dim X,
$$ 
Theorem~\ref{mn} follows from a stronger statement:

\begin{prop}
	There exists an invariant hyperplane section $P\in H\subset Y$ such that 
	$$(X,B+\gamma(l,a)f^*H+\bD_X)$$ has g-lc singularities.
\end{prop}

\begin{proof} Note that the case $l=0$ is impossible, since $\gmld_{f^{-1}P}(X,B+\bD_X)>0$.
	
	Suppose $l=1$. This means that $\sigma_0^\perp$ is a one dimensional
	vector subspace of $M_\R$. Let $\varphi_1\in M^{prim}\cap \sigma_0^\perp$.
	It defines a surjective homomorphism $\varphi_1\colon N\to \Z$.
	Since $U$ is the Minkowski sum of a polytope and $\sigma_0$, the $\varphi_1$-image
	 $U'\subset \R$ of $U\subset N_\R$ is a compact interval. Moreover, our
	 assumption $N\cap \Int(aU)=\emptyset$ is equivalent to $\Z\cap \Int(aU')=\emptyset$.
	 Since $aU'$ contains the origin, modulo replacing $\varphi_1$ by $-\varphi_1$, we may 
	 suppose $aU'\subset [0,1]$. Equivalently, $(X,B+a(\chi^{\varphi_1})_X+\bD_X)$ has 
	 g-lc singularities.
	 Since $f$ is a contraction and $\chi^{\varphi_1}\in \Gamma(\cO_X)$, 
	 there exists $\bar{\varphi}\in \bar{M}\cap \bar{\sigma}^\vee\setminus 0$ such that  
	$\varphi_1=\pi^*\bar{\varphi}$. Then $H=(\chi^{\bar{\varphi}})_Y$ satisfies the claim.
	
	Suppose $l \ge 2$. By assumption, $N\cap \Int(aU)=\emptyset$.
	By Lemma~\ref{ft}, there exists a surjective lattice homomorphism
	$\varphi\colon N\to \Z$ such that the interval $\varphi(U)$ has length $0<w\le \frac{l^2}{a}$.
	We apply the construction of the previous section to the data
	$f\colon (X,B+\bD_X)\to Y\ni P$, $t=a$ and $\varphi$. There are two possibilities:
	
	Case $0\in \partial \varphi(U)$. 
	Modulo replacing $\varphi$ by $-\varphi$, we may
	suppose $\varphi=[0,w]$. Then there exists an invariant hyperplane section 
	$P\in H\subset Y$ such that $(X,B+\frac{1}{w}f^*H+\bD_X)$ has g-lc singularities.
	We have 
	$$
	\frac{1}{w}\ge \frac{t}{l^2}\ge \gamma(l,t).
	$$
	
	Case $0\in \Int \varphi(U)$. Construct the rational map $\phi\colon X\dashrightarrow \bP^1$,
	resolve it to $\phi'\colon X'\to \bP^1$, and take $X'_1\isoto {\phi'}^{-1}(1)$. Construct
	the toric fibration $f_1\colon X'_1\to Y_1\ni P_1$. Choose $\lambda=\frac{1}{w}$.
	After log pullback, adjunction and scaling with the toric boundary, we obtain that 
	$\tilde{B}_{X'_1}$ is effective, 
	$-(K_{X'_1}+\tilde{B}_{X'_1}+\lambda\bD'_{X'_1})$ is $f_1$-nef and
	$\gmld_{f_1^{-1}P_1}(X'_1,\tilde{B}_{X'_1}+\lambda\bD'_{X'_1})\ge \lambda t=\frac{t}{w}$.
	
	Denote $l'=l(X'_1,\tilde{B}_{X'_1}+\lambda\bD'_{X'_1})$. Then $l'=l-1$, since 
	$\dim X'_1=\dim X-1$ and $\sigma_0\subset \varphi^\perp$ is the cone of g-lc places of 
	$(X'_1,\tilde{B}_{X'_1}+\lambda\bD'_{X'_1})$ as well. By induction on $l$, there exists 
	an invariant hyperplane $P_1\in H_1\subset Y_1$ such that 
	$(X'_1,\tilde{B}_{X'_1}+\gamma_1{f_1}^*H_1+\lambda  \bD_{X'_1})$ has g-lc
	singularities, with $\gamma_1=\gamma(l-1,\frac{t}{w})$.
	By Lemma~\ref{exthyp}, there exists 
	an invariant hyperplane $P\in H\subset Y$ such that 
	$(X,B+\gamma_1 f^*H+\bD_X)$ has g-lc singularities. We conclude from
	$$
	\gamma(l-1,\frac{t}{w})\ge \gamma(l-1,\frac{t^2}{l^2})=\gamma(l,t).
	$$
\end{proof}

%%%%%%%%%%%%%%%%%%%%%%%%%%%%%%%%%%%%%%%
%%%%%%%%%%%%%%%%%%%%%%%%%%%%%%%%%%%%%%%

\end{document}